\documentclass{amsart}

\usepackage{etex}
\usepackage{amsmath, amssymb}
\usepackage{array}
\usepackage[frame,cmtip,arrow,matrix,line,graph,curve]{xy}
\usepackage{graphpap, color, paralist, pstricks}
\usepackage[mathscr]{eucal}
\usepackage[pdftex]{graphicx}
\usepackage[pdftex,colorlinks,backref=page,citecolor=blue]{hyperref}
\usepackage{pifont}
\usepackage{cancel}
\usepackage{mathtools}
\usepackage{tikz}

\newtheorem{thm}{Theorem}[section]
\newtheorem{prop}[thm]{Proposition}
\newtheorem{cor}[thm]{Corollary}

\newtheorem{lemma}[thm]{Lemma}
\theoremstyle{definition}

\newcommand{\pr}{\mathbb{P}}

\newcommand{\beq}[1]{\begin{equation}\label{#1}}
\newcommand{\enq}[0]{\end{equation}}

\newcommand{\mn}[0]{\medskip\noindent}
\newcommand{\nin}[0]{\noindent}

\newcommand{\sub}[0]{\subseteq}
\newcommand{\sm}[0]{\setminus}
\renewcommand{\dots}[0]{,\ldots,}

\newcommand{\ra}[0]{\rightarrow}

\newcommand{\II}[0]{{\bf I}}

\newcommand{\TT}[0]{{\bf t}}

\newcommand{\ga}[0]{\alpha }

\newcommand{\gd}[0]{\delta }

\newcommand{\gO}[0]{\Omega}

\newcommand{\gz}[0]{\zeta}
\newcommand{\eps}[0]{\varepsilon }
\newcommand{\vt}[0]{\vartheta}

\newcommand{\comments}[1]{}

\newcommand{\mis}[0]{{\rm mis}}
\newcommand{\im}[0]{{\rm im}}
\newcommand{\itm}[0]{{\rm itm}}
\newcommand{\supp}[0]{{\rm supp}}
\newcommand{\irr}[0]{{\rm irr}}

\newcommand{\bxi}[0]{\boldsymbol\xi}
\newcommand{\bpsi}[0]{\boldsymbol\psi}
\newcommand{\tX}[0]{\tilde{X}}

\newcommand{\Ss}[0]{{\bf s}}

\newcommand{\E}[0]{{\mathbb{E}}}

\begin{document}

\title{Stability for Maximal Independent Sets}

\author{Jeff Kahn \and Jinyoung Park}
\thanks{The authors are supported by the National Science Foundation under Grant Award DMS1501962}
\thanks{JK is supported by a Simons Fellowship.}
\email{jkahn@math.rutgers.edu, jp1324@math.rutgers.edu}
\address{Department of Mathematics, Rutgers University \\
Hill Center for the Mathematical Sciences \\
110 Frelinghuysen Rd.\\
Piscataway, NJ 08854-8019, USA}

\begin{abstract}
Answering questions of Y. Rabinovich, we prove ``stability" versions of 
upper bounds on maximal independent set counts in graphs under various 
restrictions.  Roughly these say that being close to the maximum 
implies existence of a large induced matching or triangle matching
(depending on assumptions).
\end{abstract}

\maketitle


\section{Introduction}

Denote the number of maximal independent sets in a graph $G$ by $\mis(G)$. We recall two well-known bounds for these numbers:

\begin{thm}[Moon-Moser \cite{MM}] \label{thm:mm}
For any $n$-vertex graph $G$,
\[ \mis(G) \le 3^{n/3},\]
with equality iff $G$ is the disjoint union of $n/3$ triangles.
\end{thm}

\begin{thm}[Hujter-Tuza \cite{HT}]\label{thm:ht}
For any $n$-vertex, triangle-free graph $G$,
\[\mis(G) \le 2^{n/2},\]
with equality iff $G$ is a perfect matching.
\end{thm}

As usual $M$ is an \textit{induced matching} of $G$ if it is an induced subgraph of $G$ that is a matching. Similarly, $T$ is an \textit{induced triangle matching} of $G$ if it is an induced subgraph of $G$ that is a vertex disjoint union of triangles.

Write $\itm(G)$ for the number of triangles in a largest induced triangle matching in $G$, and $\im(G)$ for the number of edges in a largest induced matching.

In what follows we will usually prefer to work with $\log\mis$ ($\log=\log_2$), 
thought of as the number of bits needed to specify a maximal independent set. 
Note that
$\itm(G)\log 3$ and $\im(G)$ are obvious lower bounds on $\log\mis(G)$.
We will be interested in questions suggested to us a few years ago by 
Yuri Rabinovich \cite{Yuri} concerning ``stability" aspects of upper bounds on mis,
meaning, roughly: does large mis imply existence of a large induced triangle matching 
or large induced matching (as appropriate)?
Formally, his conjectures were unquantified versions of the following three statements,
whose proofs are the content of the present work.
(The questions were motivated by \cite{RTV}, which includes a proof of Theorem~\ref{thm:im} for
bipartite graphs.)

\begin{thm}\label{thm:itm}
For any $\eps>0$, there is a $\gd=\gd(\eps)=\gO(\eps)$ such that for an $n$-vertex graph $G$, if $\itm(G) < (1-\eps)\frac{n}{3}$ then $\log\mis(G)<(\frac{1}{3}\log3-\gd)n$.
\end{thm}

\begin{thm}\label{thm:im}
For any $\eps>0$, there is a $\gd=\gd(\eps)=\gO(\eps)$ such that for a triangle-free $n$-vertex graph $G$, if $\im(G) < (1-\eps)\frac{n}{2}$ then $\log\mis(G)<(\frac{1}{2}-\gd)n$.
\end{thm}

Theorem \ref{thm:im} applies to bipartite graphs, of course. If $G$ is bipartite with bipartition $X \cup Y$, then $\log\mis(G)$ is trivially at most $\min\{|X|,|Y|\}$ (since a maximal independent set is determined by its intersection with either of $X$, $Y$); so the statement is uninteresting unless $G$ is close to balanced. But Rabinovich asked whether something analogous also holds for unbalanced (bipartite) $G$; more precisely, whether something along the following lines is true.

\begin{thm}\label{thm:unbal}
For any $\eps>0$, there is a $\gd=\gd(\eps)=2^{-O(1/\eps)}$ such that for a bipartite graph G on $X \cup Y$ with $|X|=n$ and $|Y|=2n$, if $\im(G) < (1-\eps)n$ then $\log\mis(G)<(1-\gd)n$.
\end{thm}

\noindent
The proof of this is easily adapted to $|Y|=Bn$ (with $\delta$ then $\delta(\eps, B))$, but to keep things simple we just state the result for $B=2$.

Rabinovich suspected that, as in Theorems \ref{thm:itm} and \ref{thm:im}, $\delta(\eps)$ should be linear in $\eps$, but this is not true. In fact, Theorem \ref{thm:unbal} is tight (up to the implied constant); a construction to show this will be given in Section \ref{sec:tightness}.

\medskip
The rest of the paper is organized as follows. Section \ref{sec:prelim} recalls some background, in particular F\"uredi's upper bounds on mis for paths and cycles, and Shearer's entropy lemma. Section \ref{sec:itm,im pf} gives the proofs of Theorems \ref{thm:itm} and \ref{thm:im}. The proof of Theorem \ref{thm:unbal} and the example to show its tightness are given in Section \ref{sec:unbal pf}.

The proofs of Theorems \ref{thm:itm} and \ref{thm:im} are similar, while that of Theorem \ref{thm:unbal} is related but somewhat trickier. The general approach has its roots in an idea for counting (ordinary) independent sets due to A.A. Sapozhenko \cite{Sap2001}, \cite{Sap2007}.

Strictly speaking we prove the theorems only for sufficiently large $n$, since we occasionally hide minor terms in $o(1)$'s. Of course combined with the characterizations of equality in Theorems \ref{thm:mm} and \ref{thm:ht} this does give the stated versions, though the $\delta$'s we produce may not be valid for small $n$. Since we are really interested in large $n$ anyway, this approach seems preferable to carrying explicit error terms.

\mn
\textit{Notation.} We use ``$\sim$'' for adjacency, $N(x)$ for the neighborhood of $x$, $N(S)=\cup_{x \in S} N(x)$, and $d_S(x)=|N(x) \cap S|$. As usual, $G[S]$ is the subgraph of $G$ induced by $S \sub V(G)$.

\section{Preliminaries}\label{sec:prelim}

For the proof of Theorem \ref{thm:im} we need the following upper bounds on mis for paths and cycles, given by Z. F\"uredi \cite{Furedi}.

\begin{prop}\label{prop:Furedi}
Let $\gamma \; (\approx 1.325)$ be the unique real solution of the equation $1+\gamma=\gamma^3$.
\begin{enumerate}
\item For $P_n$, the path with $n$ vertices,
\[
\mis(P_n) \le 2 \gamma^{n-2}.
\]

\item For $C_n$, the cycle with $n$ vertices,
\[
\mis(C_n) \le 3 \gamma^{n-3}.
\]
\end{enumerate}
\end{prop}

We very briefly recall a few entropy basics (see also e.g. \cite{McE}).

For discrete random variables $X$, $Y$, the (binary) entropy of $X$ is
\[ H(X)=\sum_x p(x)\log {\frac{1}{p(x)}},\]
\nin and \textit{conditional entorpy of $X$ given $Y$} is
\[ H(X|Y)=\sum_y p(y) \sum_x p(x|y)\log\frac{1}{p(x|y)}\]
\nin (where $p(x)=\pr(X=x)$ and $p(x|y)=\pr(X=x|Y=y)$).

\begin{lemma}\label{ent}~
\begin{enumerate}
\setlength\itemsep{.5em}
\item[{\rm (a)}]
$H(X) \le \log |\mbox{Range}(X)|$, with equality iff $X$ is uniform from its range;

\item[{\rm (b)}]
$H(X,Y)=H(X)+H(Y|X)$.
\end{enumerate}
\end{lemma}

\nin
In addition to these very basic properties we need the following version of Shearer's Lemma \cite{Sh}.

\begin{lemma}\label{lem:Sh}
If $\psi=(\psi_1\dots \psi_m)$ is a random vector and $\ga:2^{[m]}\ra \mathbb R_{\ge0}$
satisfies
\beq{fractiling}
\sum_{A\ni i}\ga_A =1 \quad \forall i\in [m],
\enq
then
\beq{Shearer}
H(\psi)\leq \sum_{A\sub [m]}\ga_AH(\psi_A)
\enq
(where $\psi_A=(\psi_i:i\in A)$).
\end{lemma}

Finally, we will need the following standard fact (see e.g. Lemma 16.19 in \cite{FG}; this is also implied by Lemma \ref{lem:Sh} with $\alpha_A$ equal to $1$ if $|A|=1$ and zero otherwise).
\begin{prop}\label{prop:binom}
For $k \le \frac{1}{2}n$,
\[ \sum_{i=0}^k {n \choose i} \le 2^{H(\frac{k}{n})n}.\]
\end{prop}

\section{Proofs of Theorems \ref{thm:itm} and \ref{thm:im}}\label{sec:itm,im pf}

In this section, $I$ is always a maximal independent set in $G$.

\subsection{Algorithm}\label{alg}

Fix an order ``$\prec$'' on $V(G)$. For a given maximal independent set $I$, let $X_0=V(G)$ and repeat for $i=1,2,\ldots$:

\begin{enumerate}\setlength\itemsep{.5em}
\item Let $x_i$ be the first vertex of $X_{i-1}$ in $\prec$ among those with largest degree in $X_{i-1}$.

\item If $x_i \in I$ then let $X_i=X_{i-1} \setminus (\{x_i\} \cup N(x_i))$; otherwise, let $X_i=X_{i-1} \setminus \{x_i\}$.

\item Terminate the process if $d_{X_i}(x) \le 2$ for all $x \in X_i$.
\end{enumerate}

Let $X^*=X^*(I)=X_t$ be the final $X_i$, $t(I)=t$, and $G^*=G^*(I)=G[X^*]$. Define the sequence $\xi=\xi(I)=(\xi_1,\xi_2,\cdots, \xi_t)$ by $\xi_i := \mathbf 1_{\{x_i \in I\}}$. Notice that $\xi$ encodes a complete description of the run of the algorithm (so we may also write $G^*=G^*(\xi)$), including, in particular, the identities of the $x_i$'s. Finally, let $s=s(I)=|\supp(\xi)|$.

\subsection{Proof of Theorem \ref{thm:itm}}\label{sec:itmpf}

The argument for Theorem \ref{thm:itm} goes roughly as follows. Noting that 
\begin{equation} \label{xi}
\xi(I) \mbox{ determines both } X^* \mbox{ and } I \setminus X^*,
\end{equation}
\noindent and
\begin{equation} \label{GcapI}
I \cap X^* \mbox{ is a maximal independent set of } G^*(I) \; (=G^*(\xi)),
\end{equation}
\noindent we find that
\begin{equation} \label{missum}
\mis(G)=\sum_\xi\mis(G^*(\xi))
\end{equation}
(where the sum runs over \textit{possible} $\xi$'s).

It turns out that running the algorithm for very long is ``expensive'' in the sense that the loss in $|X^*|$, and so in possibilities for $I \cap X^*$, outweighs what is contributed to (\ref{missum}) by possibilities for $\xi$; this limits the number of $I$'s with $t(I)$ large. Similarly, the difference between the bounds in Theorems \ref{thm:mm} and \ref{thm:ht} says there are ``few'' $I$'s for which the triangle-free part of $G^*$ is large. (Note $G^*$, having maximum degree at most two, is a disjoint union of triangles and a triangle-free part, below called $R$.) But the part of $\mis(G)$ corresponding to $I$'s for which both $t$ and $R$ are small must come mainly from counting choices for the restriction of $I$ to the triangles of $G^*$, and these are limited by our assumption on $\itm(G)$.

\medskip

To begin with, the following lemma bounds the number of $I$'s with large $t(I)$.

\begin{lemma}\label{lem:nobigt} Let $\alpha=-\log(4\cdot3^{-4/3}) ~(\approx 0.113)$. For any $x \in [0,1]$,
\begin{equation}\label{nobigt}
\log|\{I:t(I)\ge xn\}| \le (\frac{1}{3}\log3-\alpha x +o(1))n.
\end{equation}
\end{lemma}

\begin{proof}

For given $t$ and $s$, consider $I$'s for which $t(I)=t$ and $s(I)=s$. Note that for each such $I$, $|V(G^*)|\le n-(t+3s)$, so by Theorem \ref{thm:mm} we have $\mis(G^*) \le 3^{(n-(t+3s))/3}$. Also, there are at most $t \choose s$ possibilities for $\xi(I)$, so by (\ref{xi}) and (\ref{GcapI}) we have
\begin{equation}\label{ineq:ts}
|\{I:t(I)=t, s(I)=s\}| \le {t \choose s}3^{(n-(t+3s))/3},
\end{equation}
\noindent so
\[ |\{I:t(I)=t\}| \le \sum_{s=0}^t {t \choose s} 3^{(n-(t+3s))/3} = 3^{n/3} \alpha_1^t,\]
\noindent where $\alpha_1=4\cdot3^{-4/3}$. Thus,
\[ |\{I:t(I) \ge xn\}| \le 3^{n/3} \alpha_1^{xn}/(1-\alpha_1),\]
\noindent yielding (\ref{nobigt}).\end{proof}

Let $T=T(I)$ be the union of the triangles in $X^*$ (so the unique maximal induced triangle matching in $G^*$), $R=R(I)=G^*[X^* \setminus V(T)]$, and $r=r(I)=|V(R)|$.
Note that there are no edges between $V(T)$ and $V(R)$, since $G^*$ has maximum degree at most 2, so
\begin{equation}\label{eq:misGTR}
\mis(G^*)=\mis(T)\mis(R).
\end{equation}
Note also that $R$ is triangle-free, so
\begin{equation}\label{ineq:misR}
\log \mis(R) \le r/2 \end{equation} by Theorem \ref{thm:ht}. Now, the following lemma bounds the number of $I$'s with large $r$.

\begin{lemma}\label{lem:nobigr}
Let $\beta=-\log(2^{1/2}3^{-1/3}) ~~(\approx 0.028)$. For any $y \in [0,1]$,
\begin{equation}\label{nobigr} \log|\{I:r(I) \ge yn\}| \le (\frac{1}{3}\log3-\beta y +o(1))n.\end{equation}
\end{lemma}

\begin{proof}
By (\ref{eq:misGTR}) and (\ref{ineq:misR}), we have
\[
|\{I:r(I) =r, t(I)=t, s(I)=s\}| \le {t \choose s} 3^{(n-(t+3s+r))/3}2^{r/2},
\]
\noindent so
\begin{eqnarray*}
|\{I:r(I)=r\}| &\le& \sum_{t=0}^n\sum_{s=0}^t {t \choose s} 3^{(n-(t+3s+r))/3}2^{r/2} \\
&\le& 3^{n/3} \beta_1^r /(1-\alpha_1), \end{eqnarray*}
\noindent where $\alpha_1=4\cdot3^{-4/3}$ (as in Lemma \ref{lem:nobigt}) and $\beta_1=2^{1/2}3^{-1/3}$. Thus,
\[ |\{I:r(I) \ge yn\}| \le 3^{n/3} \beta_1^{yn} / ((1-\alpha_1)(1-\beta_1)), \]
\noindent which gives (\ref{nobigr}).\end{proof}

\begin{lemma} \label{lem:smalltr}
If $\itm(G)<(1-\eps)n/3$ then for any $x,y \in [0,1]$,
\begin{equation} \label{smalltr}
\log|\{I:t(I)<xn, r(I)<yn\}| \le ((1-\eps)\frac{1}{3}\log3+x+y/2+o(1))n.
\end{equation} \end{lemma}

\begin{proof}
For any $I$, with $G^*=G^*(I)$ and $r=r(I)$, we have (using (\ref{eq:misGTR}), (\ref{ineq:misR}) and $|V(T(I))|=3\itm(G^*)<(1-\eps)n$)

\[ \mis(G^*) \le 3^{(1-\eps)n/3}2^{r/2}.\]
\noindent Therefore,
\[ |\{I: t(I)=t, r(I)=r\}| \le 2^t 3^{(1-\eps)n/3}2^{r/2}, \]
\noindent so
\begin{eqnarray}
|\{I:t(I)<xn,r(I)<yn\}| & \le& \sum_{t<xn}\sum_{r<yn} 3^{(1-\eps)n/3}2^{r/2+t}\nonumber \\
& \le& 3^{(1-\eps)n/3}\cdot 2^{xn+1}\cdot (\sqrt 2-1)^{-1} 2^{(yn+1)/2}, \nonumber 
\end{eqnarray}
\noindent giving (\ref{smalltr}).
\end{proof}

\begin{proof}[Proof of Theorem \ref{thm:itm}]
This is now just a matter of combining the above bounds. With $\delta_1=\eps\alpha/8$ and $\delta_2=\eps\beta/4$, Lemmas \ref{lem:nobigt} - \ref{lem:smalltr} give (respectively)
\[
\log|\{I:t(I)\ge \delta_1 n/\alpha\}| \le (\frac{1}{3}\log3-\delta_1+o(1)) n,
\]
\[
\log|\{I:r(I) \ge \delta_2n/\beta\}|\le (\frac{1}{3}\log3-\delta_2+o(1))n
\]
\noindent and (using $\frac{1}{3}\log 3 > 1/2$)
\[
\log|\{I:t(I)<\delta_1n/\alpha, r(I)<\delta_2 n /\beta\}| \le
(\frac{1}{3}\log3-\eps/4+o(1))n.
\]
\noindent Thus, with $\delta=\min\{\delta_1, \delta_2, \eps/4\}$ $(=\gO(\eps))$, we have
\[
\log\mis(G) \le (\frac{1}{3}\log3-\delta+o(1))n.
\]\end{proof}

\subsection{Proof of Theorem \ref{thm:im}}

In this section, $G$ is triangle-free. As mentioned earlier, the argument here is very similar to the one in Section \ref{sec:itmpf}, so we will try to be brief. Again, we start from the algorithm in Section \ref{alg}, and continue to use the notation ($X^*$, $G^*$ etc.) defined in the paragraph following the description of the algorithm.


\begin{lemma}\label{lem:im2}
Let $\alpha=-\log(\frac{1}{\sqrt2}+\frac{1}{4})~ (\approx 0.063)$. For any $x \in [0,1]$,
\begin{equation}\label{bigtbd}\log|\{I:t(I)\ge x n\}| \le (\frac{1}{2}-\alpha x + o(1))n.\end{equation}
\end{lemma}

\begin{proof}
Arguing as for (\ref{ineq:ts}) in Section \ref{sec:itmpf}, we obtain
\begin{equation}\label{ineq:ts1}
|\{I:t(I)=t, s(I)=s\}|\le {t \choose s}2^{(n-(t+3s))/2},
\end{equation}
\noindent where we used $\mis(G^*) \le 2^{(n-(t+3s))/2}$, as given by Theorem \ref{thm:ht} (since $G$ is triangle-free). Thus
\[
|\{I:t(I)=t\}| \le \sum_{s=0}^t{t\choose s} 2^{(n-(t+3s))/2} = 2^{n/2}\alpha_1^t,
\]
\noindent where $\alpha_1=\frac{1}{\sqrt2}+\frac{1}{4}$, and
\[
|\{I:t(I)\ge xn\}| \le 2^{n/2}\alpha_1^{xn}/(1-\alpha_1),
\]
\noindent yielding (\ref{bigtbd}).
\end{proof}

Say an edge $vw$ of $G^*$ is \textit{isolated} if $G^*[\{v,w\}]$ is a component of $G^*$. Let $M=M(I)$ be the set of isolated edges in $G^*$, $R=R(I)=G^*[X^* \setminus V(M)]$, and $r=r(I)=|V(R)|$. Since there are no edges between $V(M)$ and $V(R)$,
\begin{equation}\label{eq:misGMR}
\mis(G^*)=\mis(M)\mis(R).
\end{equation}

Note that $R$ is triangle-free, so is a vertex-disjoint union of isolated vertices, cycles with at least $4$ vertices, and paths with at least $3$ vertices. Combining this with Proposition \ref{prop:Furedi}, we obtain an upper bound for $\mis(R)$. (Recall that $\gamma \approx 1.325$ was defined in Proposition \ref{prop:Furedi}.)

\begin{lemma}\label{lem:im3_1}
With $R$ and $r$ as above, $\mis(R) \le (3\gamma)^{r/4}$.
\end{lemma}

\begin{proof}
Let $l_p$ (resp. $l_c$) be the number of vertices in the union of all paths (resp. cycles) in $R$. Clearly $l_p+l_c \le r$, while the number of paths (resp. cycles) in $R$ is at most $l_p/3$ (resp. $l_c/4$). Thus
\begin{eqnarray}
\mis(R)&\le& (2/\gamma^2)^{l_p/3}(3/\gamma^3)^{l_c/4}\gamma^r \nonumber\\
& <& (3/\gamma^3)^{r/4}\gamma^r =(3\gamma)^{r/4}, \nonumber \end{eqnarray}
\nin where the first inequality is given by Proposition \ref{prop:Furedi} and the second follows from the fact that $(2\gamma^{-2})^{1/3} < (3\gamma^{-3})^{1/4}$. \end{proof}

\begin{lemma}\label{lem:im3}
Let $\beta=-\log(2^{-1/2}(3\gamma)^{1/4}) ~(\approx 0.0023)$. For any $y \in [0,1]$,
\begin{equation}\label{bigrlbd} \log|\{I:r(I) \ge yn\}| \le (\frac{1}{2}-\beta y+o(1))n. \end{equation}
\end{lemma}

\begin{proof}
By (\ref{eq:misGMR}) and Lemma \ref{lem:im3_1}, 
\[
|\{I:r(I) =r, t(I)=t, s(I)=s\}|  \le {t \choose s} 2^{(n-(t+3s+r))/2}(3\gamma)^{r/4}, \]

\noindent and summing this over $t$ and $s$ gives
\[
\begin{split}
|\{I:r(I)=r\}| & \le 2^{n/2} \beta_1^r/(1-\alpha_1),
\end{split}
\]
\noindent where $\alpha_1=\frac{1}{\sqrt2}+\frac{1}{4}$ (as in Lemma \ref{lem:im2}) and $\beta_1=2^{-1/2}(3\gamma)^{1/4}$. Thus,
\[
|\{I:r(I) \ge yn\}| \le 2^{n/2}\beta_1^{yn}/((1-\alpha_1)(1-\beta_1)),
\]
which gives (\ref{bigrlbd}).
\end{proof}

\begin{lemma} \label{lem:im4}
If $\im(G)<(1-\eps)n/2$ then for any $x,y \in [0,1]$,
\begin{equation}\label{smtrbd}\log|\{I:t(I)<xn, r(I)<yn\}| \le ((1-\eps)/2+x+(\log(3\gamma)/4)y+o(1))n.\end{equation}
\end{lemma}

\begin{proof}
As in the proof of Lemma \ref{lem:smalltr} (using $|V(M)|\le 2\im(G^*)<(1-\eps)n$), 
\[ \mis(G^*)\le  2^{(1-\eps)n/2} (3\gamma)^{r/4}\]
\noindent for any $I$ with $r(I)=r$. Thus
\[
\begin{split}
|\{I:t(I)=t, r(I)=r\}| &\le 2^t 2^{(1-\eps)n/2}(3\gamma)^{r/4},
\end{split}
\]
\noindent and summing over the relevant $t$'s and $r$'s gives
\[
|\{I:t(I)<xn,r(I)<yn\}| \le 2^{(1-\eps)n/2}\cdot2^{xn+1} \cdot ((3\gamma)^{1/4}-1)^{-1}(3\gamma)^{(yn+1)/4};
\]
\noindent so we have (\ref{smtrbd}).
\end{proof}

\begin{proof}[Proof of Theorem \ref{thm:im}]

With $\delta_1=\eps\alpha/8$ and $\delta_2=\eps\beta/(2\log(3\gamma))$, Lemmas \ref{lem:im2}, \ref{lem:im3} and \ref{lem:im4} give (respectively)
\[
\log|\{I:t(I)\ge \delta_1 n/\alpha\}| \le (\frac{1}{2}-\delta_1+o(1))n,
\]
\[
\log|\{I:r(I) \ge \delta_2 n/\beta\}|\le (\frac{1}{2}-\delta_2+o(1))n,
\]
and
\[
\log|\{I:t(I)<\delta_1 n /\alpha, r(I)<\delta_2 n/\beta\}| \le
(\frac{1}{2}-\eps/4+o(1))n.
\]
\noindent Thus, with $\delta=\min\{\delta_1, \delta_2, \eps/4 \}$, we obtain
\[
\log\mis(G) \le (\frac{1}{2}-\delta+o(1))n.
\]
\end{proof}

\section{Proof of Theorem \ref{thm:unbal}}\label{sec:unbal pf}

For a bipartite graph $G$ on $X \cup Y$, say $X'\sub X$ is {\em irredundant}
if $\forall x\in X'$, $N(x)\not\sub N(X'\sm \{x\})$. (So for this discussion ``irredundant'' sets are always subsets of $X$.) Denote the number of irredundant sets in $G$ by $\irr(G)$. 

\begin{prop}\label{prop:irr}
For any $G$ as above, $\mis(G) \le \irr(G)$.
\end{prop}

\begin{proof}
This follows from the observation that for each maximal independent set $I$ there is an irredundant set $J \subseteq I \cap X$ with $N(J)=N(I \cap X) ~ (=Y \setminus I)$; namely, this is true whenever $J \subseteq I \cap X$ is minimal with $N(J)=N(I \cap X).$
\end{proof}

Thus the following statement implies Theorem \ref{thm:unbal}.

\begin{thm}\label{thm:unbal_irr}
For any $\eps>0$, there is a $\delta=\delta(\eps)=2^{-O(1/\eps)}$ such that for a bipartite graph $G$ on $X \cup Y$ with $|X|=n$ and $|Y|=2n$, if $\im(G) < (1-\eps)n$ then $\log\irr(G)<(1-\gd)n$.
\end{thm}

\nin For the rest of this section, $G$ is as in Theorem \ref{thm:unbal_irr}.

\subsection{Proof}

The algorithm we use for Theorem \ref{thm:unbal_irr} is slightly different from the one in section \ref{alg}. In what follows, $I$ is always an irredundant set (thus $I \subseteq X$).

\mn
\textbf{Algorithm} Let $X_0=X$, $Y_0=Y$ and $M=M_\eps=12/\eps$. Fix an order ``$\prec$'' on $X$. For a given $I$, repeat for $i=1,2,\ldots$:

\begin{enumerate}\setlength\itemsep{.5em}
\item Let $x_i$ be the first vertex of $X_{i-1}$ in $\prec$ among those with largest degree in $X_{i-1}$.
\item If $x_i \in I$ then set $Y_i=Y_{i-1} \setminus N(x_i)$; otherwise, set $Y_i=Y_{i-1}$. In either case, set $X_i=X_{i-1} \setminus \{x_i\}$.
\item Terminate the process if $d_{Y_i}(x) < M$ for all $x \in X_i$.
\end{enumerate}

Let $X^*=X^*(I)=X_t$ and $Y^*=Y^*(I)=Y_t$ be the final $X_i$ and $Y_i$, respectively. Set $t=t(I)$ and $G^*=G^*(I)=G[X^*\cup Y^*]$. As in section \ref{alg}, define $\xi=\xi(I)=(\xi_1,\xi_2,\cdots,\xi_t)$ by $\xi_i:=\mathbf{1}_{\{x_i \in I\}}$, and let $|\xi|$ be the length of $\xi$ (so $|\xi(I)|=t(I)$). Finally, let $s=s(I)=|\supp(\xi)|$ and define $\psi=\psi(I)=I\cap X^*$.

Notice that $I$ is determined by $(\xi, \psi)$, namely (as earlier) $I \setminus X^*$ is determined by $\xi$ (and $I \cap X^*=\psi$).


Consider a random (uniform) irredundant set $\II$. Our various parameters ($\xi, \psi, \ldots$) are then random variables, which will be denoted by $\bxi$ and so on. Since each of $\II$ and $(\bxi,\bpsi)$ determines the other and $\bxi$ determines $\TT$, we have (using parts (a) and (b) of Lemma \ref{ent})
\begin{eqnarray}
H(\II) &=& H(\bxi) + H(\bpsi|\bxi)\nonumber \\
& =& H(\TT) +H(\bxi|\TT)+ H(\bpsi|\bxi)\nonumber \\
&\leq& \log n +H(\bxi|\TT) + H(\bpsi|\bxi) \label{ineq:Hi}.
\end{eqnarray}

Notice that, by Lemma \ref{ent} (a),
\[ H(\bxi|\TT=t)\le t\]
\noindent for any $t$ and
\begin{equation}\label{psibd} H(\bpsi|\bxi=\xi)\le n-|\xi| ~ (=n-t)\end{equation}
\noindent for any $\xi$. Thus the sum of the last two terms in (\ref{ineq:Hi}) is at most
\[  \sum_t \pr(\TT=t)\left[ H(\bxi|\TT=t)+\sum_{|\xi|=t} \pr(\bxi=\xi|\TT=t)H(\bpsi|\bxi=\xi) \right]
 \le n, \] 
\noindent and we would like to somewhat improve these bounds. (Since we aim for $H(\II)<n-\gO(n)$, the $\log n$ in (\ref{ineq:Hi}) is irrelevant.) The next lemma, giving such a gain in (\ref{psibd}) when $t$ is small, is our main point.

\begin{lemma} \label{lem:Hpsixi}
For any $\xi$ with $|\xi|=t <\eps n/2$,
\[ H(\bpsi|\bxi=\xi) \le n-t-\vt n, \]
where $\vt=\vt(\eps)=2^{-O(1/\eps)}$.
\end{lemma}

\begin{proof}

Given $\xi$ as in the Lemma, set
\[
\tX=\tX(\xi)=\{x\in X^*:  N_{Y^*}(x)\sub N_{Y^*}(X^*\sm\{x\})\}.
\]
\nin We have
\[(1-\eps)n>\im(G)\ge \im(G^*) \ge n-t-|\tX|,\]
where the last inequality holds since for each $x \in X^* \setminus \tX$ there is some $y_x \in Y^*$ with $N_{X^*}(y_x)=\{x\}$, and $\{(x,y_x) : x \in X^* \setminus \tX\}$ is an induced matching of $G^*$ of size $|X^* \setminus \tX|=n-t-|\tX|$. Thus
%
%
\begin{equation}\label{tildeXlb}
|\tX|> \eps n -t > \eps n/2.
\end{equation}

For each $x \in \tX$ fix some $Z_x\sub X^*\sm \{x\}$ such that
\begin{equation}\label{z1}
N_{Y^*}(x)\sub N_{Y^*}(Z_x),
\end{equation}
\begin{equation}\label{z2}
|Z_x|<M
\end{equation}
and
\begin{equation}\label{z3}
\forall z\in X^* \quad |\{x\in \tX:z\in Z_x\}| <2 M.
\end{equation}
To see that we can do this:
For each $y\in N_{Y^*}(\tX)$ let $\Pi_y$ be a partition of $N_{X^*}(y)$ into blocks of
size 2 or 3.  (Note $y\in N_{Y^*}(\tX)$ implies $d_{X^*}(y)\geq 2$.)
Then to form $Z_x$, for each $y\in N_{Y^*}(x)$ choose one $x'\neq x$ from the block of
$\Pi_y$ containing $x$ and take $x'\in Z_x$. Note that each $x\in X^*$ has degree less than $M$ in $G^*$ (see step 3 of the algorithm), so we have (\ref{z2}) and (\ref{z3}) (and (\ref{z1}) is clear).

Let $W_x=Z_x \cup \{x\}$ ($x \in \tilde X$), and $\psi_{_A}=\psi \cap A$ for any $A \subseteq X$. Note that for each $x \in X^*$,
\begin{eqnarray}
H(\bpsi_{_{W_x}}|\bxi=\xi) &\le& \log[2^{|W_x|}-1] \label{ineq:psibound} \\
&=& |W_x|+\log(1-2^{-|W_x|}) \nonumber \\
&<& |W_x|-2^{-M}\log e \nonumber.
\end{eqnarray}
\noindent (The first inequality follows from irredundancy: we cannot have $\psi_{_{W_x}}=W_x$.)

Now aiming to use Lemma \ref{lem:Sh}, form $\ga : 2^{X^*} \rightarrow \mathbb R_{\ge 0}$ by assigning weight $1/(2M)$ to each $W_x$ (thus assigning
each set weight some multiple of $1/(2M)$, with the total weight of the sets containing any given $x'$ at most $1$ by (\ref{z3}))
and supplementing with weights on the singletons to get to
\eqref{fractiling}. Then by Lemma \ref{lem:Sh},
\begin{eqnarray} 
H(\bpsi|\bxi=\xi) & \le& \sum_{A\sub X^*}\ga_AH(\bpsi_A|\bxi=\xi) \nonumber \\
&=&\sum_{x \in \tilde X} \alpha_{_{W_x}} H(\bpsi_{_{W_x}}|\bxi=\xi)+\sum_{x \in X^*} \alpha_{\{x\}} H(\bpsi_{\{x\}}|\bxi=\xi) \label{eq:alpha}.
\end{eqnarray}

Now (\ref{ineq:psibound}) and the fact that $\alpha$ assigns total weight $|\tX|/(2M)$ to the $W_x$'s give
\[ 
\sum_{x \in \tilde X} \alpha_{_{W_x}} H(\bpsi_{_{W_x}}|\bxi=\xi) < \sum_{x \in \tilde X} \alpha_{_{W_x}}|W_x|-|\tilde X|(2M2^M)^{-1}\log e,
\] 
\noindent while the second sum in (\ref{eq:alpha}) is at most $\sum_{x \in X^*} \alpha_{\{x\}}$ (since $H(\bpsi_{\{x\}}|\bxi=\xi) \le 1$). Thus the entire bound in (\ref{eq:alpha}) is at most
\begin{eqnarray} \sum_{x \in \tilde X} \alpha_{_{W_x}} |W_x| + \sum_{x \in X^*} \alpha_{\{x\}} - |\tilde X| (2M2^M)^{-1}\log e
 &=& |X^*| - |\tX|(2M2^M)^{-1}\log e \nonumber \\
 &<& n-t - \vt n, \nonumber \end{eqnarray}

\noindent
where $\vt = (\eps/2)(2M2^M)^{-1}\log e=2^{-O(1/\eps)}$ (see (\ref{tildeXlb})) and we use
\[
\sum_{x \in \tilde X} \alpha_{_{W_x}}|W_x|+\sum_{x \in X^*} \alpha_{\{x\}}=\sum_{x \in X^*} \sum_{A \ni x} \alpha_A =|X^*|.
\]
\end{proof}

\begin{cor} \label{lem:smallt}
Let $\gz=\pr(\TT < \eps n/2)$. Then with $\vt$ as in Lemma \ref{lem:Hpsixi},
\[H(\bpsi|\bxi) \le n-\E\TT -\gz\vt n.\]
\end{cor}

\begin{proof}
Using Lemma \ref{lem:Hpsixi} and (\ref{psibd}) we have
\begin{eqnarray*}
H(\bpsi|\bxi) &=& \sum_t\sum_{|\xi|=t} \pr(\bxi=\xi)H(\bpsi|\bxi=\xi)\\
&\leq &\sum_{t < \eps n/2}\pr(\TT=t)(n-t-\vt n)+
\sum_{t \ge \eps n/2}\pr(\TT=t)(n-t)\\
&= & n-\E\TT -\gz\vt n.
\end{eqnarray*} \end{proof}

The gain for larger $t$ is easier. Noting that
\[ \Ss \le s_0:=2n/M=\eps n/6,\]
\noindent setting $H(1/3)=1-\gamma$ and using Proposition \ref{prop:binom}, we have, for any $t\ge\eps n/2$,
\[ H(\bxi|\TT=t) \le \log \sum_{s \le s_0} {t \choose s} \le H(1/3)t=(1-\gamma)t, \]
\noindent whence (recall $\zeta = \pr(\TT < \eps n/2$))
\begin{eqnarray*} H(\bxi|\TT) &\le& \sum_{t < \eps n/2} \pr(\TT=t)t + \sum_{t \ge \eps n/2} \pr(\TT=t)(1-\gamma)t \\
& \le& \E\TT-(1-\zeta)\gamma \eps n/2.\end{eqnarray*}

\noindent Finally, combining this with (\ref{ineq:Hi}) and Corollary \ref{lem:smallt} yields
\begin{eqnarray*} H(\II) &\le& \log n+n-[\zeta \vt +(1-\zeta)\gamma\eps /2]n\\
& \le& \log n + n - \vt n\end{eqnarray*}
\noindent (since the $\vt$ produced in Lemma \ref{lem:Hpsixi} is much smaller than $\gamma\eps/2$), proving Theorem \ref{thm:unbal_irr}.

\subsection{Tightness}\label{sec:tightness}

Define a bipartite graph $B_m$ on $X \cup Y = [m] \cup [2m]$ (disjoint copies, of course) as follows. 

\begin{enumerate}\setlength\itemsep{.5em}

\item If $x \in X$ and $x \le m-1$, then $x\sim y$ iff $y=x$ or $y=m-1+x$.

\item If $x=m \in X$, then $x \sim y$ iff $m \le y \le 2m-2$.

\end{enumerate}

\noindent It is easy to see that $\im(B_m)=m-1$, and $\mis(B_m)=2^{m}-1$.

\medskip
Now, for $\eps>0$ and $n$ with $1/\eps$ and $\eps n$ integers, let $G$ be the union of $\eps n$ disjoint copies of $B_{1/\eps}$.
Then $G$ is bipartite on $[n] \cup [2n]$, $\im(G)=(1-\eps)n$, and $\mis(G)=(2^{1/\eps}-1)^{\eps n}$. So,
\begin{eqnarray*}
\log\mis(G) &=& \eps n \log(2^{1/\eps}-1)\\
&=&\eps n(\frac{1}{\eps} + \log {(1-2^{-1/\eps})})\\
&=&n(1-2^{-1/\eps} \eps\log e+O(2^{-2/\eps}))
\end{eqnarray*}
\noindent (where the implied constant does not depend on $\eps$).

\mn
\textbf{Acknowledgment} We would like to thank Yuri Rabinovich for suggesting these questions.




\end{document}